\pdfoutput=1

\documentclass[11pt,a4paper,reqno]{amsart}

\makeatletter
\renewcommand{\paragraph}{%
  \@startsection{paragraph}{4}%
  {\z@}{\bigskipamount}{-1em}%
  {\normalfont\normalsize\bfseries}%
}
\makeatother

\usepackage[english]{babel}	

\usepackage{kotex}

\usepackage[hyphens]{url}
\usepackage[hidelinks,hyperindex,breaklinks]{hyperref} 
    \hypersetup{
        breaklinks=true,
        colorlinks=false,
        pdfusetitle=true, 
    }

\usepackage{xifthen}

\usepackage{caption} 

\usepackage[normalem]{ulem} 

\usepackage[shortlabels]{enumitem}

\usepackage{float}

\usepackage[backend=biber,style=authoryear,autocite=inline,maxcitenames=2]{biblatex}

    \usepackage{csquotes} 

\usepackage{amsmath}

\usepackage{mathtools} 

\usepackage{leftidx}

\usepackage{tikz-cd}

\usepackage{amsthm}

\makeatletter
\renewenvironment{proof}[1][\proofname]{\par
  \normalfont \topsep6\p@\@plus6\p@\relax
  \trivlist
  \item[\hskip\labelsep
        \itshape
    #1\@addpunct{.}]\ignorespaces
}{%
  \endtrivlist\@endpefalse
}
\makeatother

\newtheorem{theorem}{Theorem}[section]
\newtheorem{``theorem''}[theorem]{``Theorem''}
\newtheorem{lemma}[theorem]{Lemma}
\newtheorem{proposition}[theorem]{Proposition}
\newtheorem{corollary}[theorem]{Corollary}

\theoremstyle{definition}
\newtheorem{definition}[theorem]{Definition}
\newtheorem{``definition''}[theorem]{``Definition''}
\newtheorem{construction}[theorem]{Construction}
\newtheorem{terminology}[theorem]{Terminology}
\newtheorem{example}[theorem]{Example}
\newtheorem{remark}[theorem]{Remark}

\usepackage[stable]{footmisc} 

\newcommand{\rdsh}{\mathrel{\reflectbox{\rotatebox[origin=c]{270}{$\drsh$}}}}
\newcommand{\lush}{\mathrel{\reflectbox{\rotatebox[origin=c]{90}{$\drsh$}}}}

\newcommand{\veq}{{\rotatebox[origin=c]{90}{=}}}

\usepackage{mathabx} 

\newcommand{\defeq}{\stackrel{\text{def.}}{=}}

\newcommand{\xto}[1]{\stackrel{#1}{\to}}

\newcommand{\id}{\ensuremath{{ \mathrm{id} }}}
\newcommand{\Ob}{{ \mathrm{Ob} }}
\newcommand{\ev}{{ \mathrm{ev} }}

\newcommand{\Cat}{{ \textit{Cat} }}
\newcommand{\CAT}{{ \textsc{Cat} }}

\newcommand{\Prof}{{ \textit{Prof} }}
\newcommand{\PROF}{{ \textsc{Prof} }}

\newcommand{\Cone}{\ensuremath{{ \mathrm{Cone} }}}

\def\colim{\qopname\relax m{colim}}

\newcommand{\C}{{ \mathcal{C} }}

\newcommand{\X}{\ensuremath{{ \mathcal{X} }}}

\newcommand{\U}{{ \mathcal{U} }}
\newcommand{\V}{{ \mathcal{V} }}
\newcommand{\W}{{ \mathcal{W} }}

\newcommand{\op}{{ \mathrm{op} }}
\newcommand{\co}{{ \mathrm{co} }}

\newcommand{\Hom}{\mathrm{Hom}}

\newcommand{\Set}{{ \textit{Set} }}
\newcommand{\SET}{{ \textsc{Set} }}

\renewcommand{\:}{\colon\allowbreak}

\newcommand{\ph}{{(-)}}

\usepackage{adjustbox}

\RequirePackage{tikz}
\usetikzlibrary{svg.path}

\definecolor{orcidlogocol}{HTML}{000000} 
\tikzset{
  orcidlogo/.pic={
    \fill[orcidlogocol] svg{M256,128c0,70.7-57.3,128-128,128C57.3,256,0,198.7,0,128C0,57.3,57.3,0,128,0C198.7,0,256,57.3,256,128z};
    \fill[white] svg{M86.3,186.2H70.9V79.1h15.4v48.4V186.2z}
                 svg{M108.9,79.1h41.6c39.6,0,57,28.3,57,53.6c0,27.5-21.5,53.6-56.8,53.6h-41.8V79.1z M124.3,172.4h24.5c34.9,0,42.9-26.5,42.9-39.7c0-21.5-13.7-39.7-43.7-39.7h-23.7V172.4z}
                 svg{M88.7,56.8c0,5.5-4.5,10.1-10.1,10.1c-5.6,0-10.1-4.6-10.1-10.1c0-5.6,4.5-10.1,10.1-10.1C84.2,46.7,88.7,51.3,88.7,56.8z};
  }
}

\usepackage{scalerel}

\newcommand{\orcidurl}[1]{\href{https://orcid.org/#1}{\scalerel*{%
\begin{tikzpicture}[yscale=-1,transform shape]%
\pic{orcidlogo};%
\end{tikzpicture}
}{\texttt{|}}\,\texttt{https://orcid.org/#1}}}

\makeatletter
\def\@setthanks{\def\thanks##1{\par##1}\thankses}
\makeatother

\title{Indexed profunctors over 2-categories}

\author{Sori Lee}

\thanks{\emph{Date}: 13 Feb 2023. \orcidurl{0000-0002-3911-8909}}

\begin{document}

\maketitle

\begin{abstract}
We define the notion of an indexed profunctor over a 2-category, and use it to develop an abstract theory of limits.
The theory subsumes (conical) limits, weighted limits, ends and Kan extensions.
Results include an abstract version of the theorem that right adjoint functors preserve limits, and an abstract account of the phenomenon that the comparison arrow implicated in the preservation of a limit by a functor is natural in any functorial variable the functor happens to depend on.
These results make an extensive use of the data and axioms of an indexed profunctor over a 2-category.
\end{abstract}

\tableofcontents

\section{Introduction}

These notes define indexed profunctors over 2-categories and initiate their study.
The focus is on their use for an abstract theory of limits.
The theory subsumes (conical) limits, weighted limits, ends and Kan extensions.

In general terms, indexed profunctors are the indexed version of profunctors, just as indexed categories are of categories.
Indexed profunctors over 1-categories appear in a few places in the literature.
\textcite{wood_abstract_1982} makes a brief mention of them as an example of proarrows.
\textcite{koudenburg_algebraic_2012} uses indexed profunctors over the particular 1-category $(0 \rightrightarrows 1)$ for describing weak variants of double categories as corresponding weak algebras for the free-strict-double-category monad.
\textcite{shulman_enriched_2013} defines an enriched version of indexed profunctors over 1-categories while developing the theory of enriched indexed categories.
Indexed profunctors over 2-categories appear unexplored.

Here is an outline of how indexed profunctors over 2-categories naturally arise as a framework for abstract limits.
The starting point is ordinary profunctors.
Recall that if $R$ and $S$ are categories, then a \emph{profunctor} $R \to S$ is a functor $R^\op \times S \to \SET$.
Given a profunctor $H\: R \to S$ and an object $s \in S$, we say that an object $r \in R$ together with an element $h \in H(r,s)$ is a \emph{limit} of $s$ if the pair $(r,h)$ is a universal element of the functor $H(-,s)\: R^\op \to \SET$.
Now, if $I$ and $X$ are categories, we may consider a profunctor $\Cone_X\: X \to \CAT(I,X)$
given by $\Cone_X(x,d) = \CAT(I,X)(\Delta_x,d)$.
Clearly, a limit of an object $d \in \CAT(I,X)$ in this profunctor is precisely a limiting cone of the diagram $d$.
We will see that this profunctorial notion of a limit subsumes not just conical limits but also the other limit concepts mentioned above.

Plain profunctors, however, are soon met with limitations.
As we attempt to generalise some key theorems about (say conical) limits, we find ourselves in need of more structure.
Consider for instance the theorem stating that right adjoint functors preserve conical limits.
To formulate an abstraction of this statement in terms of profunctorial limits, we need to address not a single profunctor but a family $H_X\: R_X \to S_X$ of profunctors equipped with an appropriate functorial action in $X \in \CAT$.
If $\CAT$ here is just considered a 1-category, then such a family is an indexed profunctor over a 1-category.
Even in this setting, we can formulate the property of a functor (or more generally a 1-cell) preserving a profunctorial limit and hence obtain an abstract version of the theorem's statement.
However, the 1-categorical indexing lacks the necessary information to enable its proof.

To rectify this, we need that $R_\ph$ and $S_\ph$ act functorially on not just functors (1-cells) but also natural transformations (2-cells).
Moreover, crucially, these actions must be compatible with the functorial action of $H_\ph$.
With the addition of these requirements, we arrive at the notion of an indexed profunctor over a 2-category.
The 2-categorical setting notably enables us to also abstract away from the specific base $\CAT$, and state the theorem in terms of adjoint 1-cells in an arbitrary 2-category.
But most importantly, it allows us to prove it.

The `adjoints preserve' theorem (Section \ref{sec:right-adjoints-preserve-limits}) is one of the two abstract theorems proved in these notes that necessitate the data and axioms of an indexed profunctor over a 2-category.
The other theorem gives an abstract account of the phenomenon that the comparison arrow implicated in the preservation of a limit by a functor is natural in any functorial variable the functor happens to depend on (Section \ref{sec:parametrised-preservation}).
It was the problem that prompted this work.
The two, though small in number, are offered as exemplary uses of indexed profunctors over 2-categories as a framework for abstract limits.

On the other hand, we will also come across abstract results about limits for which the structure of an indexed profunctor is not necessary.
These will be stated and proved in terms of ordinary profunctors, and (when desirable) indexed-profunctorial formulations will be derived as corollaries.
Examples of such results include the functoriality of limit (Section \ref{sec:functoriality-of-limit}) and a profunctorial analysis of the fact that fully faithful functors reflect limits (Section \ref{sec:fully-faithful-reflect}).

Some essential topics are not addressed in the present form of these notes.
Currently, the justification of the definition of an indexed profunctor over a 2-category is bottom-up and rests solely on its applications in abstract limits.
In order to complement this, further structural examination of the notion is required.
To begin with, a systematic relationship between indexed categories, functors and prfunctors over a 2-category should be established in terms of a framework such as proarrow equipment.
Another undiscussed subject worthy of mention is morphisms of indexed profunctors; to be more precise, morphisms between pairs of indexed categories over a 2-category together with an indexed profunctor between them.
Well-known results in category theory allow reducing one limit notion (e.g. ends) to another (e.g. conical limits).
Morphisms attract attention because a natural class of them describes the structure of such reductions.

These notes develop the theory of indexed profunctors between strictly indexed categories over a strict 2-category.
While the adoption of the fully weak setting of pseudoindexed categories over a bicategory is unlikely to affect the general scheme of the theory, the pseudoindexing in particular does introduce additional technicalities, including in the definition of an indexed profunctor.
For this exploratory work, it seemed reasonable to take advantage of the simplicity of the strict setting as it nevertheless covers the main examples.

Further generalisations could be made in various directions.
One inviting direction is enrichment.
We may for instance seek a common generalisation of this work and Shulman's, and define indexed monoidal categories, enriched indexed categories and enriched indexed profunctors over a 2-category.
Another direction is to use an abstract substitute for profunctors such as proarrows.
Finally, we may go higher.
The experiences from this work anticipate that indexed 2-profunctors over 3-categories would be needed for an effective abstract theory of 2-limits, indexed $(\infty,1)$-profunctors over $(\infty,2)$-categories for abstract $(\infty,1)$-limits, and so forth.
How do these theories work out?

\section{Profunctorial preliminaries}

In this section, we discuss a few ideas central in dealings with ordinary (i.e.\ non-indexed) profunctors.
These will in particular be needed for our first discussions of indexed profunctors in Section \ref{sec:def}, and remain important in all later sections.

Before we discuss profunctors, 
a clarification is in order on how we will deal with sizes of sets.

\paragraph{Convention on sizes}

In these notes, we will use three \emph{relative} sizes of sets as well as structures based on sets: \emph{small}, \emph{medium} and \emph{big}\footnote{Adjective `large' is avoided because it has established meanings that are incompatible with ours such as `non-small'. In contrast, typical notions of `small' (e.g.\ `a set as opposed to a proper class') are deemed largely compatible.}.
The relativity means that we are unspecific about these sizes other than that if $\U$, $\V$ and $\W$ denoted the respective universes in order, then $\U < \V < \W$.
We shall furthermore employ the \emph{polymorphism} convention that these universes are regarded as universally quantified over and over in each breakable context, rather than fixed throughout.

Whenever size matters, structures (such as sets, categories, etc.) shall be tacitly assumed medium unless otherwise specified.
In fact, the definition of a structure with no size specification shall technically mean to define its medium version and tacitly also, by polymorphism, all other relatively small versions (i.e.\ small, big, etc.)
of the structure.

Because of the polymorphism, the choice of a size for an entity being introduced doesn't make any difference in meaning as long as it is not contradictory with respect to the size choices of other entities in the local discourse.
The rule of thumb will be to make a choice that minimises explicit mentions of sizes in the story being told.
The explicit mention of universes will also be avoided unless necessary.
This is to keep the size matter as low-profile as possible.

Symbols for totalities of structures will follow the following pattern.
A totality of small structures will take the italic typeface.
For example, $\Set$ for the category of small sets.
A totality of medium structures will take the small-caps typeface.
For example, $\SET$ for the big category of sets.

\paragraph{Pulling back a profunctor along functors}

Recall that a profunctor $H\: R \to S$ is a functor $R^\op \times S \to \SET$.
In a profunctorial context, we shall refer to $R$ and $S$ as the \emph{domain} and \emph{codomain} of $H$ respectively.

An important operation on a profunctor is to pull it back along functors into its domain and codomain categories.
Specifically, the \emph{pullback} of a profunctor $H\: R \to S$ along functors $f\: R' \to R$ and $g\: S' \to S$ is a profunctor $H(f,g)\: R' \to S'$, defined to be the composite functor
\begin{equation}\label{eq:gyodae}
\begin{tikzcd}[column sep=large]
    R'^\op \times S' \arrow[r, "f^\op \times g"] & R^\op \times S \arrow[r, "H"] & \SET.
\end{tikzcd}
\end{equation}
 
\begin{remark}
One should beware not to confuse this profunctor $H(f,g)\: R' \to S'$ with the function $H(a,b)\: H(r_2,s_1) \to H(r_1,s_2)$ for arrows $a\: r_1 \to r_2$ in $R$ and $b\: s_1 \to s_2$ in $S$.
Note that $H(f,g)$ is essentially -- up to the pedantic `op' -- the multicategorical notation for the composite \eqref{eq:gyodae}, which itself should be drawing on the notation for values of the composite functor: $H(f,g)(r',s') = H(f(r'),g(s'))$ when $r' \in R'$ and $s' \in S'$.
This is a convenient notation, especially compared to $H \circ (f^\op \times g)$, for its intuitiveness and simplicity, and seems in common use in proarrow theory \autocite{nlab:2-category_equipped_with_proarrows}.
\end{remark}

By extension, we can also pull back a natural transformation of profunctors along natural transformations of functors.
Let $H,I\: R \to S$ be profunctors and $\theta\: H \to I$ a natural transformation of profunctors.
Furthermore, let $f_1,f_2\: R' \to R$ and $g_1,g_2\: S' \to S$ be functors and let $\alpha\: f_1 \to f_2$ and $\beta\: g_1 \to g_2$ be natural transformations.
Then the \emph{pullback} of $\theta$ along $\alpha$ and $\beta$ is a natural transformations of profunctors $\theta(\alpha,\beta)\: H(f_2,g_1) \to I(f_1,g_2)$, defined to be the horizontal composite of natural transforamtions
\begin{equation*}
\begin{tikzcd}[column sep=huge]
    R'^\op \times S'
        \arrow[r, bend left=30, "f_2^\op \times g_1", ""{below, name=U}]
        \arrow[r, bend right=30, "f_1^\op \times g_2" swap, ""{name=D}]
        \arrow[Rightarrow, from=U, to=D, "\alpha^\op \times \beta"]
        &
    R^\op \times S
        \arrow[r, bend left=30, "H", ""{below, name=U2}]
        \arrow[r, bend right=30, "I" swap, ""{name=D2}]
        \arrow[Rightarrow, from=U2, to=D2, "\theta"]
        &
    \SET.
\end{tikzcd}
\end{equation*}

These pullbacks will often feature in the discussions, particularly the definition, of indexed profunctors.

\paragraph{Heteromorphisms and heteromorphic diagrams}

\begin{definition}
Let $H\: R \to S$ be a profunctor between categories and let $r \in R$ and $s \in S$ be objects.
An element $h \in H(r,s)$ may be called a \emph{heteromorphism}\footnote{This terminology seems to have been introduced by \textcite{ellerman_theory_2006}.} from $r$ to $s$.
We may write $h\: r \to s$ to indicate that $h$ is from $r$ to $s$.
\end{definition}

The terminology and notation are suggesting that we view elements in a profunctor as a sort of arrows.
In fact, heteromorphisms are a generalisation of arrows, in that the arrows in a category $C$ are precisely the heteromorphisms in the particular profunctor $\Hom_C\: C \to C$.
Importantly, part of this generalisation is the composition:
the pre- and postcomposition of an arrow in $C$ with appropriate composable arrows in $C$ are a special case of the contra- and covariant actions of a profunctor $H\: R \to S$ on its heteromorphisms against appropriate arrows in $R$ and $S$ respectively.
This way, we view the heteromorphisms as `composing' with arrows in $R$ and $S$.

A consequence of this view is a generalised diagrammatic reasoning.
We can consider diagrams in which some of the constituent arrows are heteromorphisms, and make sense of their commutativity.
Such diagrams may be referred to as \emph{heteromorphic diagrams}.
For example, we can consider squares of the form
$$\begin{tikzcd}
r \arrow[r, "h"]
  \arrow[d, "a"] &
s \arrow[d, "b"] \\
r' \arrow[r, "h'"] &
s'
\end{tikzcd}$$
with $h\: r \to s$ and $h'\: r' \to s'$ heteromorphisms in $H$, $a\: r \to r'$ an arrow in $R$ and $b\: s \to s'$ an arrow in $S$.
The \emph{commutativity} of this square refers to the equality $a^*(h') = b_*(h)$ of the two heteromorphisms from $r$ to $s'$.

The benefits of using heteromorphic diagrams are as usual: the pictorial presentation can make an equality more intuitive, oftentimes highlighting a symmetry that is present in the `type' of the equality.
They will have a marked presence throughout these notes.

\paragraph{Profunctorial duality}

Opposite categories are a fundamental operation on categories underpinning categorical duality.
The following defines an operation on profunctors that generalises opposite categories, which in turn will underpin a profunctorial generalisation of categorical duality.

\begin{definition}
Let $H\: R \to S$ be a profunctor.
The profunctor \emph{opposite} to $H$ is the profunctor $H^\op\: S^\op \to R^\op$ given by
$$H^\op(s,r) = H(r,s)$$
for all objects as well as arrows $s \in S$ and $r \in R$.
\end{definition}

Opposite profunctors generalise opposite categories in that
    $\Hom_{C^\op} = \Hom_C^\op$
if $C$ is a category.
Taking the profunctorial opposite is still a strict involution:

\begin{proposition}\label{prop:op-is-involution}
Let $H\: R \to S$ be a profunctor.
Then $(H^\op)^\op = H$.\qed
\end{proposition}

Let us now describe the profunctorial duality in terms of opposite profunctors.
The duality may be organised into a metamathematical construction and theorem, as follows.
I have chosen to treat them informally (specifically, I won't make precise what a `statement' is); they so still serve as sufficiently practical a meta-algorithm and justificiation for applications in these notes.

\begin{``definition''}
A statement $\phi(H,r,s,h)$ about an arbitrary heteromorphism $h\: r \to s$ in an arbitrary profunctor $H\: R \to S$ has an associated \emph{dual} statement
    $$\phi^\op(H,r,s,h) := \phi(H^\op,s,r,h)$$
still about an arbitary heteromorphism $h\: r \to s$ in an arbitrary profunctor $H\: R \to S$.
\end{``definition''}

\begin{``theorem''}[Profunctorial duality principle]
Let $\phi(H,r,s,h)$ be as above.
Then $\forall H,r,s,h.\phi(H,r,s,h)$ if and only if $\forall H,r,s,h.\phi^\op(H,r,s,h)$.
\end{``theorem''}

\begin{proof}[``Proof'']
`Only if':
If $\forall H,r,s,h.\phi(H,r,s,h)$, then for any $H,r,s,h$, we have $\phi(H^\op,s,r,h)$, that is, $\phi^\op(H,r,s,h)$.
`If':
If $\forall H,r,s,h.\phi^\op(H,r,s,h)$, that is, $\forall H,r,s,h.\phi(H^\op,s,r,h)$, then for any $H,r,s,h$, we have a profunctor $H^\op\: S \to R$ and a heteromorphism $h\: s \to r$ in $H^\op$, so we have $\phi((H^\op)^\op,r,s,h)$, that is, by Proposition \ref{prop:op-is-involution}, $\phi(H,r,s,h)$, as desired.
\end{proof}

We will encounter examples of the profunctorial duality throughout these notes.
In fact, we are to see one immediately as we next discuss profunctorial limits and colimits.

\paragraph{Limits and colimits in profunctors}

Limits and colimits in profunctors as we will now define are an abstraction of traditional `conical' limits and colimits in categories as well as other (co)limit-like concepts such as weighted (co)limits, (co)ends and Kan extensions, as we will see in Section \ref{sec:eg} when we discuss examples of indexed profunctors.

\begin{definition}
Let $H\: R \to S$ be a profunctor.
We say that a heteromorphism $h\: r \to s$ in $H$ is a \emph{limit} if the pair $(r,h)$ is a universal element of the functor $H(-,s)\: R^\op \to \SET$.
\end{definition}

In this case, we say that the object $s$ is \emph{convergent} or \emph{has a limit} in $H$, and refer to the pair $(r,h)$ as a limit \emph{of} $s$ in $H$.
We say that $H$ \emph{has} limits if each object $s \in S$ has a limit.

Dually, a heteromorphism $h\: r \to s$ in $H$ is a \emph{colimit} if it is a limit as a heteromorphism $s \to r$ in the profunctor $H^\op$.
That is, if the pair $(s,h)$ is a universal element of the functor $H(r,-)\: S \to \SET$.
In this case, we say that the object $r$ is \emph{convergent} or \emph{has a colimit} in $H$, and refer to the pair $(s,h)$ is a colimit \emph{of} $r$ in $H$.
We say that $H$ \emph{has} colimits if each object $r \in R$ has a colimit.

\section{Definition}\label{sec:def}

In what follow, a 2-category shall refer to a strict 2-category, a 2-functor a strict 2-functor and a 2-natural transformation a strict 2-natural transformation.

Recall that an indexed category over a category $C$ is usually defined to be a pseudofunctor $C^\op \to \CAT$, where $\CAT$ denotes the big 2-category of categories.
We need a generalisation of (the strict and covariant version of) this notion in which the base is a 2-category.
I state its definition in a certain verbose fashion, in order to motivate the definition of an indexed profunctor below.

\begin{definition}
A \emph{strict covariant indexed category} $R$ over a 2-category $\X$ consists of
\begin{enumerate}
    \item a category $R_X$ for each 0-cell $X \in \X$,
    \item a functor $R_f\: R_X \to R_Y$ for each 1-cell $f\: X \to Y$ in $\X$, and
    \item a natural transformation $R_\theta\: R_f \to R_g$ for each 2-cell $\theta\: f \to g$ in $\X$
\end{enumerate}
such that the assignment $X \mapsto R_X$ is strictly 2-functorial.

In these notes, we will simply refer to a strict covariant indexed category as an \emph{indexed category}.
\end{definition}

In crisp language, such an indexed category is simply a 2-functor $\X \to \CAT$.

Before we define indexed profunctors between indexed categories, let us first define indexed functors.

\begin{definition}
Let $R$ and $S$ be indexed categories over a 2-category $\X$.
A \emph{strict indexed functor} $k\: R \to S$ consists of a functor $k_X\: R_X \to S_X$ for each 0-cell $X \in \X$ that is strictly natural in $X$.

In these notes, we will simply refer to a strict indexed functor as an \emph{indexed functor}.
\end{definition}

In crisp language, such an indexed functor is simply a 2-natural transformation $R \to S$.

The following is the main definition of these notes.

\begin{definition}
Let $R$ and $S$ be indexed categories over a 2-category $\X$.
An \emph{indexed profunctor} $H\: R \to S$ consists of
\begin{enumerate}
    \item a profunctor $H_X\: R_X \to S_X$ for each 0-cell $X \in \X$, and
    \item a natural transformation $H_f\: H_X \to H_Y(R_f,S_f)$ for each 1-cell $f\: X \to Y$ in $\X$
\end{enumerate}
such that $H_X$ is `functorial' in $X$ and $H_f$ is extranatural in $f$.
\end{definition}

A clarification of the last two axioms is in order.
First, the `functoriality' means the following two things.
\begin{enumerate}[(i)]
    \item\label{item:i} That $H_{\id_X}\: H_X \to H_X(R_{\id_X},S_{\id_X}) = H_X$ equal $\id_{H_X}\: H_X \to H_X$ as natural transformations of profunctors for each object $X \in \X$.
    
    \item\label{item:ii} That $H_{g \circ f}\: H_X \to H_Z(R_{g \circ f},S_{g \circ f}) = H_Z(R_g \circ R_f, S_g \circ S_f)$ equal the (vertical) composite
        $$H_g(R_f,S_f) \circ H_f\: H_X \to H_Z(R_g,S_g)(R_f,S_f) = H_Z(R_g \circ R_f, S_g \circ S_f)$$
    as natural transformations of profunctors for each 1-cells $f\: X \to Y$ and $g\: Y \to Z$ in $\X$.
\end{enumerate}
The extranaturality on the other hand means ordinary extranaturality: that the diagram
\begin{equation}\label{eq:extranaturality-axiom}
\begin{tikzcd}[column sep=large]
H_X \arrow[r, "H_f"]
    \arrow[d, "H_{f'}"] &
H_Y(R_f,S_f)
    \arrow[d, "{H_Y(R_f,S_\theta)}"] \\
H_Y(R_{f'},S_{f'})
    \arrow[r, "{H_Y(R_\theta,S_{f'})}"] &
H_Y(R_f,S_{f'})
\end{tikzcd}
\end{equation}
commute for each 2-cell $\theta\: f \to f'$ between 1-cells $f,f'\: X \to Y$ in $\X$.

\begin{remark}
The axioms of an indexed profunctor, of which the functoriality conditions have been particularly arduous to spell out above, can be represented more naturally using \emph{double-categorical diagrams of profunctors}.

In such a diagram, a node represents a category, a vertical arrow between nodes a functor between categories, a horizontal arrow between nodes a profunctor between categories, a double arrow between vertical arrows a natural transformation between functors, a double arrow between horizontal arrows a natural transformation between profunctors.
What makes it different from a usual 2-categorical diagram is that a 1-cellular `niche' of the form
    $\begin{tikzcd}[scale=0.5]
    A \arrow[d, "f" swap] &
    B \arrow[d, "g"] \\
    C \arrow[r, "H" swap] &
    D
    \end{tikzcd}$
represents the horizontal arrow $H(f,g)\: A \to B$ and a 2-cellular niche of the form
    $\begin{tikzcd}[scale=0.5]
    A \arrow[d, bend right=90, "f" swap, ""{name=L1}]
      \arrow[d, "f'", ""{left, name=L2}]
      \arrow[phantom, from=L1, to=L2, "\stackrel{\alpha}{\Rightarrow}"] &
    B \arrow[d, "g" swap, ""{name=R1}]
      \arrow[d, bend left=90, "g'", ""{left, name=R2}]
      \arrow[phantom, from=R1, to=R2, "\stackrel{\beta}{\Rightarrow}"]
    \\
    C \arrow[r, "H" swap] &
    D
    \end{tikzcd}$
represents the double arrow between horizontal arrows $H(\alpha,\beta)\: H(f',g) \to H(f,g')$.

In terms of the double-categorical diagrams of profunctors, the functoriality condition \ref{item:i} is the equality
\begin{center}
\begin{tikzcd}
R_X \arrow[r, "H_X"]
    \arrow[d, "R_{\id_X}" swap]
    \arrow[rd, phantom, "{\Downarrow}{\scriptstyle H_{\id_X}}"] &
S_X \arrow[d, "S_{\id_X}"] \\
R_X \arrow[r, "H_X" swap] &
S_X
\end{tikzcd}    
$=$
\begin{tikzcd}
R_X \arrow[r, "H_X"]
    \arrow[d, equal]
    \arrow[rd, phantom, "{\Downarrow}{\scriptstyle \id_{H_X}}"] &
S_X \arrow[d, equal] \\
R_X \arrow[r, "H_X" swap] &
S_X,
\end{tikzcd}
\end{center}
and condition \ref{item:ii} the equality
\begin{center}
\begin{tikzcd}
R_X \arrow[r, "H_X"]
    \arrow[dd, "R_{g \circ f}" swap]
    \arrow[rdd, phantom, "{\Downarrow}{\scriptstyle H_{g \circ f}}"] &
S_X \arrow[dd, "S_{g \circ f}"] \\ \\
R_Z \arrow[r, "H_Z" swap] &
S_Z
\end{tikzcd}    
$=$
\begin{tikzcd}
R_X \arrow[r, "H_X"]
    \arrow[d, "R_f" swap]
    \arrow[rd, phantom, "{\Downarrow}{\scriptstyle H_f}"] &
S_X \arrow[d, "S_f"] \\
R_Y \arrow[r, "H_Y" swap]
    \arrow[d, "R_g" swap]
    \arrow[rd, phantom, "{\Downarrow}{\scriptstyle H_g}"] &
S_Y \arrow[d, "S_g"] \\
R_Z \arrow[r, "H_Z" swap] &
S_Z.
\end{tikzcd}
\end{center}
The extranaturality condition \eqref{eq:extranaturality-axiom} is the equality
\begin{center}
\begin{tikzcd}
R_X \arrow[rr, "H_X"]
    \arrow[d, bend right=90, "R_f" swap, ""{right, name=L}]
    \arrow[d, "R_{f'}", ""{left, name=R}]
    \arrow[r, phantom, from=L, to=R, "\stackrel{R_\theta}{\Rightarrow}"]
    \arrow[rrd, phantom, "{\Downarrow}{\scriptstyle H_{f'}}"] &&
S_X \arrow[d, "S_{f'}"] \\
R_Y \arrow[rr, "H_Y" swap] &&
S_Y
\end{tikzcd}
$=$
\begin{tikzcd}
R_X \arrow[rr, "H_X"]
    \arrow[d, "R_f" swap]
    \arrow[rrd, phantom, "{\Downarrow}{\scriptstyle H_f}"] &&
S_X \arrow[d, "S_f" swap, ""{right, name=L}]
    \arrow[d, bend left=90, "S_{f'}", ""{left, name=R}]
    \arrow[r, phantom, from=L, to=R, "\stackrel{S_\theta}{\Rightarrow}"] \\
R_Y \arrow[rr, "H_Y" swap] &&
S_Y.
\end{tikzcd}
\end{center}
\end{remark}

\begin{remark}\label{rem:extranaturality-pointwise}
The extranaturality condition \eqref{eq:extranaturality-axiom} comes down to something modest also when examined pointwise.
Namely that whenever $h\: r \to s$ is a heteromorphism in $H_X$, then the heteromorphic diagram
$$\begin{tikzcd}
f_*r \arrow[r, "f_*h"] 
     \arrow[d, "\theta_*r" swap] &
f_*s \arrow[d, "\theta_*s"] \\
f'_*r \arrow[r, "f'_*h"] &
f'_*s
\end{tikzcd}$$
in $H_Y$ commute (``the heteromorphism $f_*h\: f_*r \to f_*s$ is natural in $f$'').
\end{remark}

We can obtain a crisp description of an indexed profunctor by shifting the formulational burden onto an abstraction, just as in the case of an indexed category or an indexed functor, as follows.

\begin{definition}\label{def:diagonal-section}
Let $\X$ be a 2-category and $E_\ph\: \X^{\co,\op} \times \X^\op \to \CAT$ a 2-functor.
A \emph{diagonal section} $H$ of $E$ consists of
\begin{enumerate}
    \item an object $H_X \in E_{X,X}$ for each object $X \in \C$, and
    \item an arrow $H_f\: H_X \to E_{f,f}(H_Y)$ in $E_{X,X}$ for each 1-cell $f\: X \to Y$ in $\C$
\end{enumerate}
such that $H_X$ is functorial in $X$ and the morphism $H_f\: H_X \to E_{f,f}(H_Y)$ is extranatural in $f$.
\end{definition}

Now, if $R$ and $S$ are indexed categories over a 2-category $\X$, then an indexed profunctor $H\: R \to S$ that is \emph{small}, i.e.\ one for which each $H_X$ is a small profunctor, is precisely a diagonal section of the 2-functor
\begin{align*}
    \X^{\co,\op} \times \X^\op &\to \CAT \\
    (X,X') &\mapsto \Prof(R_X,S_{X'}),
\end{align*}
where $\Prof(R_X,S_{X'})$ denotes the category of small profunctors $R_X \to S_{X'}$.

Let us introduce a terminology for the existence of limits in an indexed profunctor, which generalises the established phraseology in the case of limits in categories.

\begin{terminology}
Let $H\: R \to S$ be an indexed profunctor over a 2-category $\X$.
We shall say an object $X \in \X$ \emph{has (co)limits (of type $H$)} if the profunctor $H_X$ has (co)limits.
\end{terminology}

\section{Examples}\label{sec:eg}

\begin{example}[Hom sets]
We will render hom sets in categories as an indexed endoprofunctor on the identity 2-functor $\id_\CAT\: \CAT \to \CAT$ regarded as an indexed category.

For each category $X \in \CAT$, we have a profunctor $\Hom_X\: X \to X$.
If $f\: X \to Y$ is a functor in $\CAT$, then we have a natural transformation
    $$\Hom_f\: \Hom_X = \Hom_X(-,-) \to \Hom_Y(f-,f-) = \Hom_Y(f,f)$$
given by the functorial action of $f$, which is clearly functorial in $f$.
Finally, it is easy to see that the extranaturality condition at a natural transformation $\theta\: f \to f'$ between functors $f,f'\: X \to Y$ in $\CAT$ simply amounts to the naturality of $\theta$.
Therefore $\Hom_{(-)}$ is an indexed profunctor $\id_\CAT \to \id_\CAT$.

Let us examine limits and colimits in this indexed profunctor.
If $X \in \CAT$ is a category and $x \in X$ is an object, then any isomorphism $x' \xto{\cong} x$ is a limiting heteromorphism into $x$, and any isomorphism $x \xto{\cong} x'$ is a colimiting heteromorphism from $x$.
In particular, the identity arrow $\id_x\: x \to x$ is both limiting and colimiting as a heteromorphism.
\end{example}

\begin{example}[Conical limits in categories]
Let $I \in \CAT$ be a category.
Consider the 2-functors $\id_\CAT,\CAT(I,-)\: \CAT \to \CAT$, regarded as indexed profunctors over $\CAT$.
We will define an indexed profunctor $H\: \id_\CAT \to \CAT(I,-)$ in which limits are conical limits of shape $I$ in categories.

If $X \in \CAT$ is a category, $x \in X$ an object and $d \in \CAT(I,X)$ a functor, we define
$$
    H_X(x,d) := \CAT(I,X)(\Delta_x,d),
$$
where $\Delta_x\: I \to X$ is the constant-$x$ functor.
This 2-hom set\footnote{A \emph{2-hom} set in a 2-category $\X$ denotes a set of the form $\X(X,Y)(f,g)$.} is medium, since every 2-hom set in $\CAT$ is.
It is obvious from this formula what its functorial actions in $x$ and $d$ shall be.

If $f\: X \to Y$ is a functor in $\CAT$, we define the natural transformation
$H_f\: H_X \to H_Y(f,\CAT(I,f))$
to be the one given by whiskering:
$$\begin{array}{rcl}
    (H_f)_{x,d}\: \CAT(I,X)(\Delta_x,d) & \to & {\CAT(I,Y)(\Delta_{f(x)},f \circ d)} \\
    & & \hspace{4.5em}\veq \\
    & & \CAT(I,Y)(f \circ \Delta_x,f \circ d) \\
    \gamma & \mapsto & f \ast \gamma.
\end{array}$$
Note that both naturalities of $H_f$ in $d$ and $x$ follow from the interchange law of natural transformations, whereas the functoriality of $H$ holds by the associativity of horizontal composition of natural transformations.

Finally, let us verify the extranaturality of $H_f$ in $f$.
Let $\theta\: f \to f'$ be a natural transformation between functors $f,f'\: X \to Y$.
We need to see that the diagram
\begin{center}\begin{tikzcd}[column sep=8em]
    {\CAT(I,X)(\Delta_x,d)}
        \arrow[r, "{f\ast-}"]
        \arrow[d, "f'\ast-"] & {\CAT(I,Y)(f \circ \Delta_x,f \circ d)} \arrow[d, "{\CAT(I,Y)(f \circ \Delta_x, \theta \ast d)}"] \\
    {\CAT(I,Y)(f' \circ \Delta_x,f' \circ d)} \arrow[r, "{\CAT(I,Y) (\theta \ast \Delta_x, f' \circ d)}"] & {\CAT(I,Y)(f \circ \Delta_x,f' \circ d)}
\end{tikzcd}\end{center}
commutes.
Let $\gamma \in \CAT(I,X)(\Delta_x,d)$.
By the interchange law of natural transformations,
$${\drsh}(\gamma) = (f' \ast \gamma) \circ (\theta \ast \Delta_x) = \theta \ast \gamma  = (\theta \ast d) \circ (f \ast \gamma) = {\rdsh}(\gamma)$$
as desired.
This completes the construction of $H$.
\end{example}

The following example is in substance a duplicate of the previous one, with just $\CAT$ replaced with an arbitrary 2-category $\X$.

\begin{example}[Conical limits in 0-cells in a 2-category]
Let $\X$ be a 2-category with a strict 2-terminal object.
Let $I$ be a 0-cell in $\X$.
We will define indexed categories $R$ and $S$ over $\X$ and an indexed profunctor $H\: R \to S$, such that when $\X = \Cat$, the 2-category of small categories, limits in $H$ are conical limits of shape $I$ in small categories.

Take $R := \X(1,-)$ and $S := \X(I,-)$, which are 2-functors $\X \to \CAT$, hence indexed categories over $\X$.

If $X \in \X$ is a 0-cell, and $x \in \X(1,X)$ and $d \in \X(I,X)$ 1-cells in $\X$, then we define
$$
    H_X(x,d) := \X(I,X)(x \circ {!},d).
$$
It is obvious from this formula what its functorial actions in $x$ and $d$ shall be.

If $f\: X \to Y$ is a 1-cell in $\X$, we define the natural transformation
$H_f\: H_X \to H_Y(\X(1,f), \X(I,f))$
to be the one given by whiskering:
$$\begin{array}{rcl}
    (H_f)_{x,d}\: \X(I,X)(x \circ {!},d) & \to & {\X(I,Y)(f \circ x \circ {!},f \circ d)} \\
    \gamma & \mapsto & f \ast \gamma.
\end{array}$$
Note that both naturalities of $H_f$ in $x$ and $d$ follow from the interchange law of 2-cells, whereas the functoriality of $H$ holds by the associativity of horizontal composition of $2$-cells.

Finally, let us verify the extranaturality of $H_f$ in $f$.
Let $\theta\: f \to f'$ be a 2-cell between 1-cells $f,f'\: X \to Y$.
We need to see that the diagram
\begin{center}\begin{tikzcd}[column sep=8em]
    {\X(I,X)(x \circ {!},d)}
        \arrow[r, "{f\ast-}"]
        \arrow[d, "f'\ast-"] &
    {\X(I,Y)(f \circ x \circ {!},f \circ d)} 
        \arrow[d, "{\X(I,Y)(f \circ x \circ {!}, \theta \ast d)}"] \\
    {\X(I,Y)(f' \circ x \circ {!},f' \circ d)}
        \arrow[r, "{\X(I,Y)(\theta \ast x \circ {!}, f' \circ d)}"] &
    {\X(I,Y)(f \circ x \circ {!},f' \circ d)}
\end{tikzcd}\end{center}
commutes.
Let $\gamma \in \X(I,X)(x \circ {!},d)$.
By the interchange law of 2-cells,
$${\drsh}(\gamma) = (f' \ast \gamma) \circ (\theta \ast (x \circ {!})) = \theta \ast \gamma  = (\theta \ast d) \circ (f \ast \gamma) = {\rdsh}(\gamma)$$
as desired.
This completes the construction of $H$.
\end{example}

\begin{example}[Weighted limits in categories]
Let $I \in \CAT$ be a category.
We will define an indexed category $S$ and an indexed profunctor $H\: \id_\CAT \to S$ over $\CAT$ in which limits are weighted limits of shape $I$ in categories.

We take $S := \CAT(I,\Set) \times \CAT(I,-)$, which is a 2-functor $\CAT \to \CAT$ since every hom category is medium in $\CAT$.

If $X \in \CAT$ is a category, $x \in X$ an object, $w\: I \to \Set$ a functor and $d\: I \to X$ a functor, we define
$$
    H_X(x, (w,d)) := \CAT(I,\Set)(w,X(x,d-)).
$$
This 2-hom set is medium, since every 2-hom set in $\CAT$ is.
It is obvious from this formula what the functorial actions of $H_X$ in $x$, $w$ and $d$ shall be.
Therefore $H_X$ is a (medium) profunctor $X \to \CAT(I,\Set) \times \CAT(I,X)$.

If $f\: X \to Y$ is a functor between categories, we define the natural transformation
$H_f\: H_X \to H_Y(f, \CAT(I,\Set) \times \CAT(I,f))$
to be the one given by vertical postcomposition with the functorial action of $f$:
$$\begin{array}{rcl}
    (H_f)_{x,(w,d)}\: \CAT(I,\Set)(w,X(x,d-)) & \to & \CAT(I,\Set)(w,Y(fx,fd-)) \\
    \gamma & \mapsto & f_{x,d-} \circ \gamma
\end{array}$$
where $f_{x,d-}\: X(x,d-) \to Y(fx,fd-)$ is the natural transformation between the two functors $I \to \Set$ given by the functorial action of $f$.
Here, the naturality of $H_f$ in $w$ holds by the associativity of vertical composition of natural transformations, while its naturalities in $x$ and $w$ follow from the functoriality of $f$.

The functoriality of $H$ follows from the associativity of vertical composition of natural transformations.

Finally, let us verify the extranaturality of $H_f$ in $f$.
Let $\theta\: f \to f'$ be a natural transformation between functors $f,f'\: X \to Y$.
We need to see that the diagram
\begin{center}
\adjustbox{scale=0.88}{%
\begin{tikzcd}[column sep=8em]
    {\CAT(I,\Set)(w,X(x,d-))} \arrow[r, "{f_{x,d-} \circ -}"] \arrow[d, "{f'_{x,d-} \circ -}"] & {\CAT(I,\Set)(w,Y(fx,fd-))} \arrow[d, "{\CAT(I,\Set)(w,Y(fx,\theta_{d-}))}"] \\
    {\CAT(I,\Set)(w,Y(f'x,f'd-))} \arrow[r, "{\CAT(I,\Set)(w,Y(\theta_x,f'd-))}"] & {\CAT(I,\Set)(w,Y(fx,f'd-))}
\end{tikzcd}}
\end{center}
commutes.
Let $\gamma \in \CAT(I,\Set)(w,X(x,d-))$, $i \in \Ob(I)$ and $r \in w(i)$.
Then
$${\drsh}(\gamma)_i(r) = f'(\gamma_i(r)) \circ \theta_x = \theta_{d(i)} \circ f(\gamma_i(r)) = {\rdsh}(\gamma)_i(r)$$
by the naturality of $\theta$ against the arrow $\gamma_i(r)\: x \to d(i)$ in $X$, as desired.
This completes the construction of $H$.
\end{example}

\begin{example}[Ends in categories]
Let $I \in \CAT$ be a category.
We will define an indexed category $S$ over $\CAT$ and an indexed profunctor $H\: \id_\CAT \to S$ whose limits are ends of shape $I$ in categories.

We take $S := \CAT(I^\op \times I,-)$, which is a 2-functor $\CAT \to \CAT$ since every hom category in $\CAT$ is medium.

If $X \in \CAT$ is a category, $x \in X$ an object and $d\: I^\op \times I \to X$ a functor, then $H_X(x,d)$ is defined to be the set of wedges from $x$ to $d$, which is evidently medium.
This set is functorial in $x$ and $d$, as can be seen with the visual help of the figure
\begin{equation*}
\begin{tikzcd}
x' 
\arrow[rd, "r"] 
\arrow[drr, bend left=15, dashrightarrow] 
\arrow[ddr, bend right=15, dashrightarrow] \\
  & x
    \arrow[r, "w_i"]
    \arrow[d, "w_{i'}"]
    & d(i,i)
      \arrow[r, "\theta_{i,i}"]
      \arrow[d, "{d(i,a)}"]
         & d'(i,i)
           \arrow[dd, "{d'(i,a)}"] \\
   & d(i',i')
     \arrow[r, "{d(a,i')}"] 
     \arrow[d, "\theta_{i',i'}"]
     & d(i,i')
       \arrow[rd, "\theta_{i,i'}"]
       & \\
   & d'(i',i')
     \arrow[rr, "{d'(a,i')}"]
     & & d'(i,i')
\end{tikzcd}
\end{equation*}
which illustrates the precomposition of an arrow $r\: x' \to x$ and postcomposition of a natural transformation $\theta\: d \to d'$ to a wedge $w\: x \to d$ where $a\: i \to i'$ is an arrow in $I$.

If $f\: X \to Y$ is a functor between categories, we define the natural transformation
$H_f\: H_X \to H_Y(f,\CAT(I^\op \times I,f))$ to have components
$$\begin{array}{rcl}
    (H_f)_{x,d}\: H_X(x,d) & \to & H_Y(fx, f \circ d) \\
    w & \mapsto & f \ast w := (fx \xto{f(w_i)} fd(i,i) \mid i \in \Ob(I)).
\end{array}$$
The wedgness of the family $f \ast w$ as well as the naturality of $H_f$ in $x$ and $d$ follows simply from the functoriality of $f$, whereas the functoriality of $H$ is immediate from the compatibility of function composition and application.

Finally, let us verify the extranaturality of $H_f$ in $f$.
Let $\theta\: f \to f'$ be a natural transformation between functors $f,f'\: X \to Y$.
We need to see that the diagram
\begin{center}\begin{tikzcd}[column sep=8em]
    H_X(x,d)
        \arrow[r, "{f\ast-}"]
        \arrow[d, "f'\ast-"] &
    H_Y(fx, f \circ d)
        \arrow[d, "{H_Y(fx, \theta \ast d)}"] \\
    H_Y(f'x, f' \circ d)
        \arrow[r, "{H_Y(\theta_x, f' \circ d)}"] &
    H_Y(fx, f' \circ d)
\end{tikzcd}\end{center}
commutes.
Let $w \in H(X)(x,d)$ and $i \in \Ob(I)$.
Then
$${\drsh}(w)_i = f'(w_i) \circ \theta_x = \theta_{d(i,i)} \circ f(w_i) = {\rdsh}(\gamma)_i$$
by the naturality of $\theta$ at the arrow $w_i\: x \to d(i,i)$ in $X$, as desired.
This completes the construction of $H$.
\end{example}

\begin{example}[Right Kan extensions in 2-categories]\label{eg:Kan}
Let $\X$ be a 2-category and $I,A \in \X$ 0-cells.
We will define indexed categories $R,S$ over $\X$ and an indexed profunctor $H\: R \to S$ whose limits are right Kan extensions in $\X$.

Take $R := \X(A,-)$ and $S := \X(I,A) \times \X(I,-)$, which are clearly 2-functors $\X \to \CAT$.

If $X \in \X$ is a 0-cell, and $r \in \X(A,X)$, $d \in \X(I,A)$ and $k \in \X(I,X)$ 1-cells, then we define
$$
    H_X(r,(d,k)) := \X(I,X)(r \circ k,d).
$$
It is obvious from this formula what the functorial actions of $H_X$ in $r$, $d$ and $k$ shall be.

If $f\: X \to Y$ is a 1-cell in $\X$, we define the natural transformation
$H_f\: H_X \to H_Y(\X(A,f), \X(I,A) \times \X(I,f))$
to be the one given by whiskering:
$$\begin{array}{rcl}
    (H_f)_{r,(d,k)}\: \X(I,X)(r \circ k,d) & \to & \X(I,Y)(f \circ r \circ k,f \circ d) \\
    \gamma & \mapsto & f \ast \gamma.
\end{array}$$
Note that both naturalities of $H_f$ in $r$ and $(d,k)$ follow from the interchange law of 2-cells, whereas the functoriality of $H$ holds by the associativity of horizontal composition of 2-cells.

Finally, let us verify the extranaturality of $H_f$ in $f$.
Let $\theta\: f \to f'$ be a 2-cell between 1-cells $f,f'\: X \to Y$.
We need to see that the diagram
\begin{center}\begin{tikzcd}[column sep=8em]
    {\X(I,X)(r \circ k,d)}
        \arrow[r, "{f\ast-}"]
        \arrow[d, "f'\ast-"] &
    {\X(I,Y)(f \circ r \circ k, f \circ d)}
        \arrow[d, "{\X(I,Y)(f \circ r \circ k, \theta \ast d)}"] \\
    {\X(I,Y)(f' \circ r \circ k, f' \circ d)}
        \arrow[r, "{\X(I,Y)(\theta \ast (r \circ k), f' \circ d)}"] &
    {\X(I,Y)(f \circ r \circ k, f' \circ d)}
\end{tikzcd}\end{center}
commutes.
Let $\gamma \in \X(I,X)(r \circ k,d)$.
By the interchange law of 2-cells,
$${\drsh}(\gamma) = (f' \ast \gamma) \circ (\theta \ast (r \circ k)) = \theta \ast \gamma = (\theta \ast d) \circ (f \ast \gamma) = {\rdsh}(\gamma)$$
as desired.
This completes the construction of $H$.
\end{example}

\section{Preservation, reflection, lifting and creation of limits}\label{sec:preservation-etc}

The purpose of this section is to introduce the notions of \emph{preservation}, \emph{reflection}, \emph{lifting} and \emph{creation} of limits in an indexed profunctor by a 1-cell in the base 2-category.
Such a 1-cell, however, can be viewed as a special case of a morphism of (non-indexed) profunctors, and many basic results about the preservation, etc.\ of limits are already procurable at the generality of a morphism of profunctors.
For this reason, we will define the notions in question at this greater generality, whilst rendering the corresponding indexed-profunctorial notions as a special case.

Let us first define morphism of profunctors.
In what follow, we may denote a profunctor $H\: R \to S$ as a triple $(R,S,H)$.

\begin{definition}
Let $(R,S,H)$ and $(R',S',H')$ be profunctors.
A \emph{morphism} $(R,S,H) \to (R',S',H')$ is a triple $(\rho,\sigma,\eta)$ where $\rho\: R \to R'$ and $\sigma\: S \to S'$ are functors and $\eta\: H \to H'(\rho,\sigma)$ is a natural transformation.
\end{definition}

The next definition establishes the way the preservation, etc.\ of limits in an indexed profunctor by a 1-cell in the base 2-category are a special case of a morphism of profunctors preserving, etc.\ limits.

\begin{definition}\label{def:preserves-etc}
Let $H\: R \to S$ be an indexed profunctor between strict indexed categories over a 2-category $\X$.
A 1-cell $f\: X \to Y$ in $\X$ \emph{preserves}, \emph{reflects}, \emph{(strictly) lifts}, \emph{(strictly) uniquely (strictly) lifts} or \emph{(strictly) creates} (co)limits if the morphism of profunctors
    $$(R_f,S_f,H_f)\: (R_X,S_X,H_X) \to (R_Y,S_Y,H_Y(R_f,S_f))$$
does so respectively.
\end{definition}

Of course, what the latter in turn mean will now be defined one by one.

\paragraph{Preservation and reflection of limits}

\begin{definition}
A morphism $(\rho,\sigma,\eta)\: (R,S,H) \to (R',S',H')$ of profunctors \emph{preserves} (co)limits if whenever a heteromorphism $h\: r \to s$ in $H$ is a (co)limit, then the heteromorphism $\eta(h)\: \rho(r) \to \sigma(s)$ in $H'$ is a (co)limit.

Conversely, the morphism $(\rho,\sigma,\eta)$ \emph{reflects} (co)limits if whenever $h\: r \to s$ is a heteromorphism in $H$ such that the heteromorphism $\eta(h)\: \rho(r) \to \sigma(s)$ in $H'$ is a (co)limit, then $h$ is a (co)limit.
\end{definition}

Note that if $H$ is the indexed profunctor over $\CAT$ for conical (co)limits, then the resulting notions of preservation and reflection of (co)limits in $H$ are precisely the usual notions of a functor preserving and reflecting (co)limits.

\paragraph{Lifting of limits}

\begin{definition}
A morphism $(\rho,\sigma,\eta)\: (R,S,H) \to (R',S',H')$ of profunctors \emph{strictly lifts} limits if whenever $s \in S$ is an object such that $\sigma(s) \in S'$ has a limit $h'\: r' \to \sigma(s)$ in $H'$, then $s$ has a limit $h\: r \to s$ in $H$ such that $r' = \rho(r)$ and $h' = \eta(h)$.
It does so \emph{strictly uniquely} if there is moreover at most one such pair $r,h$.

Dually, the morphism $(\rho,\sigma,\eta)\: (R,S,H) \to (R',S',H')$ \emph{strictly lifts} colimits if whenever $r \in R$ is an object such that $\rho(r) \in R'$ has a colimit $h'\: \rho(r) \to s'$, then $r$ has a colimit $h\: r \to s$ in $H$ such that $s' = \sigma(s)$ and $h' = \eta(h)$.
It does so \emph{strictly uniquely} if there is moreover at most one such pair $s,h$.
\end{definition}

I'm mimicking the nomenclature of \textcite{riehl_category_2017}, \S3.3 which distinguishes the strict and non-strict variants of creation of conical limits (we will shortly also discuss creation).
The earlier literature of \textcite{adamek_abstract_1990} refers to (the conical version of) strict lifting as simply `lifting' without `strict'.

The non-strict variant of lifting can be obtained by basically replacing the use of equality of objects by the use of an isomorphism.
The resulting definition is evidently equivalent to the following perhaps shorter formulation.

\begin{definition}
A morphism $(\rho,\sigma,\eta)\: (R,S,H) \to (R',S',H')$ of profunctors \emph{lifts} limits if whenever $s \in S$ is an object such that $\sigma(s) \in S'$ has a limit, then $s$ has a limit $h\: r \to s$ in $H$ such that $\eta(h)\: \rho(r) \to \sigma(s)$ in $H'$ is a limit.

Dually, the morphism $(\rho,\sigma,\eta)\: (R,S,H) \to (R',S',H')$ \emph{lifts} colimits if whenever $r \in R$ is an object such that $\rho(r) \in R'$ has a colimit, then $s$ has a colimit $h\: r \to s$ in $H$ such that $\eta(h)\: \rho(r) \to \sigma(s)$ in $H'$ is a colimit.
\end{definition}

Note that the unstated non-strict \emph{uniqueness} condition on the lifting is vacuous, because of the universal property of (co)limits.
(We will nonetheless speak of `non-strict unique lifting' for at least once when we define non-strict creation of limits below, for reasons of symmetry with the strict case.)

We note the following trivial consequence of the lifting of (co)limits, which perhaps rhymes well with the natural language in which it is stated.

\begin{proposition}
Let $(\rho,\sigma,\eta)\: (R,S,H) \to (R',S',H')$ be a morphism between profunctors.
If $(R',S',H')$ has (co)limits and $(\rho,\sigma,\eta)$ lifts them, then $(R,S,H)$ has (co)limits.\qed
\end{proposition}

\begin{corollary}
Let $H$ be an indexed profunctor over a 2-category $\X$.
Let $X,Y$ be 0-cells in $\X$.
If $Y$ has (co)limits of type $H$ and a 1-cell $X \to Y$ lifts them, then $X$ has (co)limits of type $H$.\qed
\end{corollary}

\paragraph{Creation of limits}

\begin{definition}
A morphism of profunctors \emph{strictly creates} (co)limits if it reflects and strictly uniquely strictly lifts (co)limits.
It \emph{creates} (co)limits if it reflects and (uniquely\footnote{vacuous, as noted earlier}) lifts (co)limits.
\end{definition}

As with lifting, we are following \textcite{riehl_category_2017} in distinguishing strict and non-strict variants.
Earlier literature such as \textcite{mac_lane_categories_1998}
calls strict creation simply `creation' without `strict'.
Our definition indeed specialises to the (strict and non-strict variants of) creation of conical (co)limits as defined in \emph{op.\ cit}.

Creation \emph{conditionally} subsumes preservation:

\begin{proposition}
Let $(\rho,\sigma,\eta)\: (R,S,H) \to (R',S',H')$ be a morphism between profunctors.
If $(R',S',H')$ has (co)limits and $(\rho,\sigma,\eta)$ creates them, then $(\rho,\sigma,\eta)$ preserves them.
\end{proposition}

\begin{proof}
We will prove this for limits.
The colimit case is dual.

Let $h\: r \to s$ be a limit in $H$.
We need to show that $\eta(h)\: \rho(r) \to \sigma(s)$ is a limit in $H'$.
Since $(R',S',H')$ has limits, there exists a limit $h'\: r' \to \sigma(s)$ of $\sigma(s)$.
By lifting, there exists a limit $h_2\: r_2 \to s$ such that $\eta(h_2)\: \rho(r_2) \to \sigma(s)$ is a limit.
Since both $h$ and $h_2$ are limits of $s$, there exists an isomorphism $\phi\: r \to r_2$ in $R$ that identifies them.
It follows that $\rho(\phi)\: \rho(r) \to \rho(r_2)$ is an isomorphism in $R'$ that identifies $\eta(h)$ and $\eta(h_2)$.
Now, since $\eta(h_2)$ is a limit, so is $\eta(h)$, as desired.
This proves the proposition.
\end{proof}

\begin{corollary}
Let $H$ be an indexed profunctor over a 2-category $\X$.
Let $X,Y$ be objects in $\X$.
If $Y$ has (co)limits and a 1-cell $f\: X \to Y$ creates them, then $f$ preserves them.\qed
\end{corollary}

\section{Functoriality of limit}
\label{sec:functoriality-of-limit}

This section is again in principle about non-indexed profunctors.
Its aim is to substantiate the idea that ``taking profunctorial limit is a functor''.
Besides being of general interest, this will be needed in Section \ref{sec:parametrised-preservation} (in the form of Corollary \ref{cor:limit-functor} below).

Let us first consider a parametrised version of the functoriality, from which a non-parametrised version will follow as a special case.

\begin{lemma}[Parametrised limit functor]\label{lem:parametrised-limit-functor}
Let $H\: R \to S$ be a profunctor.
Let $C$ be a category and $\sigma\: C \to S$ a functor.
For each $c \in \Ob(C)$, let $\rho(c) \in \Ob(R)$ together with $\lambda_c \in H(\rho(c),\sigma(c))$ be a limit of $\sigma(c)$.
Then $\rho$ extends uniquely to a functor $C \to R$ such that for each arrow $a\: c' \to c$ in $C$, the heteromorphic diagram
    \begin{equation}\label{eq:limit-functor-def}
    \begin{tikzcd}
    \rho(c')
        \arrow[r, "\lambda_{c'}"]
        \arrow[d, dashrightarrow, "\rho(a)" swap] &
    \sigma(c')
        \arrow[d, "\sigma(a)"] \\
    \rho(c)
        \arrow[r, "\lambda_{c}"] &
    \sigma(c)
    \end{tikzcd}
    \end{equation}
    in $H$ commutes in the sense that $\rho(a)^*(\lambda_c) = \sigma(a)_*(\lambda_{c'})$.\qed
\end{lemma}

\begin{remark}
This lemma is a fine showcase of the effectiveness of heteromorphic diagrams.
Without them, seeing that $\rho$ respects composition would have required a technical reasoning.
But with the visual help of diagram \eqref{eq:limit-functor-def}, the fact becomes rather evident that I believe omitting the proof is a good way to convey it.
\end{remark}

\begin{corollary}[Limit functor]\label{cor:limit-functor}
Let $H\: R \to S$ be a profunctor.
For each $s \in \Ob(S)$, let $\lim{s} \in \Ob(R)$ together with $\lambda_s \in H(\lim{s},s)$ be a limit of $s$.
Then $\lim$ extends uniquely to a functor $S \to R$ such that for each arrow $a\: s' \to s$ in $S$, the diagram
    $$\begin{tikzcd}
    \lim{s'}
        \arrow[r, "\lambda_{s'}"]
        \arrow[d, dashrightarrow, "\lim{a}" swap] &
    s'
        \arrow[d, "a"] \\
    \lim{s}
        \arrow[r, "\lambda_{s}"] &
    s
    \end{tikzcd}$$
in $H$ commutes in the sense that $\lim(a)^*(\lambda_s) = a_*(\lambda_{s'})$.\qed
\end{corollary}

Next, we will note as a consequence of this functoriality that the existence of limits in a profunctor is equivalent to its `representability'.

\begin{definition}
A profunctor $H\: R \to S$ is \emph{representable} if there is a functor $\rho\: S \to R$ and an isomorphism of profunctors $H \cong R(1,\rho)$, and \emph{corepresentable} if there is a functor $\sigma\: R \to S$ and an isomorphism of profunctors $H \cong S(\sigma,1)$.
\end{definition}

It is immediate from the definition of representability that if a profunctor $H\: R \to S$ is representable, then it has all limits.
The converse will follow from the following proposition.

\begin{proposition}
Let $H\: R \to S$ be a profunctor, and let $\rho\: C \to R$ and $\sigma\: C \to S$ be functors.
For each $c \in \Ob(C)$, let $\eta_c \in H(\rho(c),\sigma(c))$.
The following are equivalent.
\begin{enumerate}
    \item The heteromorphisms $\eta_c\: \rho(c) \to \sigma(c)$ are natural in $c$. That is, the diagram
    \begin{equation*}
    \begin{tikzcd}
    \rho(c)
        \arrow[r, "\eta_{c}"]
        \arrow[d, "\rho(a)" swap] &
    \sigma(c)
        \arrow[d, "\sigma(a)"] \\
    \rho(c')
        \arrow[r, "\eta_{c'}"] &
    \sigma(c')
    \end{tikzcd}
    \end{equation*}
    in $H$ commutes for each arrow $a\: c \to c'$ in $C$.
    
    \item For any object $r \in R$, the function
    $$(-)^*(\eta_c)\: R(r,\rho(c)) \to H(r,\sigma(c))$$
    is natural in $c$.
\end{enumerate}
\end{proposition}

\begin{proof}
Let $q\: r \to \rho(c)$ be an arrow in $R$, and consider the heteromorphic diagram
    \begin{equation*}
    \begin{tikzcd}
    r \arrow[r, "q"] &
    \rho(c)
        \arrow[r, "\eta_{c}"]
        \arrow[d, "\rho(a)" swap] &
    \sigma(c)
        \arrow[d, "\sigma(a)"] \\
    &
    \rho(c')
        \arrow[r, "\eta_{c'}"] &
    \sigma(c')
    \end{tikzcd}
    \end{equation*}
in $H$.
Clearly the two ways of obtaining a heteromorphism $r \to \sigma(c')$ in this diagram coincide if the square commutes, which proves that (1) implies (2).
Conversely, if the two ways coincide, then setting $q = \id_{\rho(c)}$ shows that the square commutes, proving (2) implies (1).
This proves the proposition.
\end{proof}

\begin{corollary}\label{cor:pulling-back-lambda-is-iso}
Let $H$, $C$, $\sigma$, $\rho$ and $\lambda$ be as in Lemma \ref{lem:parametrised-limit-functor}.
Then there is a canonical isomorphism
\begin{equation*}\label{eq:rho-represents-H-sigma}
    R(1,\rho) \xto{\cong} H(1,\sigma)
\end{equation*}
of profunctors $R \to C$ given by pulling back the $\lambda_c$.
In other words, the functor $\rho\: C \to R$ represents the profunctor $H(1,\sigma)\: R \to C$.
\end{corollary}

\begin{proof}
For any objects $r \in R$ and $c \in C$, the function
$$(-)^*(\lambda_c)\: R(r,\rho(c)) \to H(r,\sigma(c))$$
is bijective because $\lambda_c$ is a limit, and is natural in $r$ because such functions always are.
It is moreover natural in $c$ by the previous proposition.
This proves the corollary.
\end{proof}

\begin{corollary}\label{cor:pulling-back-lambda-is-iso-2}
Let $H$, $\rho$ and $\lambda$ be as in Corollary \ref{cor:limit-functor}.
Then there is a canonical isomorphism
\begin{equation*}
    R(1,\rho) \xto{\cong} H
\end{equation*}
of profunctors $R \to S$ given by pulling back the $\lambda_s$.
In other words, the functor $\rho\: S \to R$ represents the profunctor $H\: R \to S$.\qed
\end{corollary}

We further observe that (parametrised) limit functors are in fact limits in a `power' profunctor.

\begin{construction}
Let $H\: R \to S$ be a profunctor and $C$ a category.
We will define a profunctor $H^C\: R^C \to S^C$ ($H$ to the \emph{power} $C$).

Note that the profunctor $H\: R \to S$ gives rise to the functor
    $$H(-,-)\: (R^C)^\op \times S^C \to \PROF(C,C)$$
and that taking the concrete end is a functor $\int_C\: \PROF(C,C) \to \SET$, where $\PROF(C,C)$ denotes the big category of profunctors $C \to C$.
Taking the composite of these two functors defines a functor
    $$H^C := \int_C H(-,-)\: (R^C)^\op \times S^C \to \SET,$$
which is the desired profunctor $H^C\: R^C \to S^C$.
\end{construction}

\begin{lemma}\label{lem:profunctor-power-jointly-reflect}
Let $H\: R \to S$ be a profunctor, $C$ a category, $\rho\: C \to R$ and $\sigma\: C \to S$ functors, and $\lambda \in H^C(\rho,\sigma) = \int_C H(\rho,\sigma)$.
If $\lambda_c \in H(\rho(c),\sigma(c))$ is a limit of $\sigma(c)$ for each $c \in \Ob(C)$, then $\lambda \in H^C(\rho,\sigma)$ is a limit of $\sigma$.
\end{lemma}

Dually, if $\lambda_c$ is a colimit of $\rho(c)$ in $H$ for each $c \in \Ob(C)$, then $\lambda$ is a colimit of $\rho$ in $H^C$.

\begin{proof}
Let $\rho'\: C \to R$ be a functor.
We need to show that the function
\begin{equation}\label{eq:banana}
\begin{aligned}
    R^C(\rho',\rho) & \to H^C(\rho',\sigma) \\
    \alpha &\mapsto \alpha^*(\lambda) := H^C(\alpha,\sigma)(\lambda) \defeq (\int_{c \in C} H(\alpha_ c,\sigma c))(\lambda)
\end{aligned}
\end{equation}
is bijective, where we can see
    $\alpha^*(\lambda)_c = H(\alpha_c,\sigma c)(\lambda_c).$

Corollary \ref{cor:pulling-back-lambda-is-iso} tells us that pulling back components of $\lambda$ gives an isomorphism $R(1,\rho) \xto{\cong} H(1,\rho)$ between profunctors $R \to C$, which restricts to an isomorphism $R(\rho',\rho) \xto{\cong} H(\rho',\sigma)$ between profunctors $C \to C$.
Applying the canonical end functor to this isomorphism in $\PROF(C,C)$, we obtain a bijection
\begin{equation}\label{eq:mango}
\begin{aligned}
    R^C(\rho',\rho) = \int_C R(\rho',\rho) &\xto{\cong} \int_C H(\rho',\sigma) = H^C(\rho',\sigma) \\
    \alpha &\mapsto \alpha^\star(\lambda)
\end{aligned}
\end{equation}
where $\alpha^\star(\lambda)_c = H(\alpha_c,\sigma c)(\lambda_c)$.
We see that functions \eqref{eq:banana} and \eqref{eq:mango} coincide by what they do.
Hence \eqref{eq:banana} is bijective.
\end{proof}

Lemmas \ref{lem:parametrised-limit-functor} and \ref{lem:profunctor-power-jointly-reflect} combine to the following summary of this section.

\begin{theorem}\label{thm:jointly-strictly-create}
The evaluation morphisms
$$(\ev_c\: (R^C,S^C,H^C) \to (R,S,H) \mid c \in \Ob{C})$$
jointly strictly create limits.\qed
\end{theorem}

Dually, they jointly strictly create colimits.

\section{Parametrised limit preservation is natural}
\label{sec:parametrised-preservation}

Let $C$ be a (locally small) category and $c \in C$ an object.
Consider the comparison arrows in the preservation of conical limits and of ends by the contravariant hom functor.
For conical limits, the arrow has the form
$$C(\lim d, c) \to \lim C(d,c)$$
where $d\: I \to C$ is a diagram.
For ends, the arrow has the form
$$C(\int_I d, c) \to \int_I C(d,c)$$
where $d\: I^\op \times I \to C$ is an end diagram.
Suppose we want to prove that both arrows are natural in $c$, in a way that corroborates the notion that the two naturalities are instances of one phenomenon.

One way to do this would be to make use of the fact that ends can be described as conical limits: prove that under this description the comparison arrow for ends is the one for conical limits, and that the latter is natural in $c$.

Another way, which we will pursue in this section, is by abstraction.
It amounts to (a little less, in terms of generality, than) proving the following theorem, as the two concrete naturalities above are clearly its special cases (with $\X = \CAT$, $X = C^\op$, $Y = \Set$ and $\phi(c) = C(-,c)$).

\begin{theorem}\label{thm:naturality-of-preservator}
Let $H\: R \to S$ be an indexed profunctor over a 2-category $\X$.
Let $C$ be a category, $X,Y \in \X$ objects and $\phi\: C \to \X(X,Y)$ a functor.
Suppose $X$ and $Y$ have limits of type $H$ and let $s \in \Ob(S)$.
Then the canonical arrow
$$\phi(c)_*(\lim{s}) \to \lim{\phi(c)_*(s)}$$
in $R_Y$ is natural in $c \in \Ob(C)$.
\end{theorem}

The consequent dual statement asserts that the canonical arrow
$$\colim \phi(c)_*(s) \to \phi(c)_*(\colim s)$$
in $S_Y$ is natural in $c \in \Ob(C)$.

In what follow, we will explore a modular proof of this theorem.
First, the following profunctorial generalisation of the notion of monomorphism will be useful.

\begin{definition}
Let $H\: R \to S$ be a profunctor.
A heteromorphism $h\: r \to s$ in $H$ is \emph{monic} if whenever $a,b\: r' \to r$ are parallel arrows in $R$ such that $a^*(h) = b^*(h)$ then $a = b$.
\end{definition}

Dually, $h$ is called \emph{epic} if whenever $a,b\: s \to s'$ are parallel arrows in $S$ such that $a_*(h) = b_*(h)$ then $a = b$.

Clearly, if $H = \Hom_C$ for a category $C$, then these definitions give traditional monomorphisms and epimorphisms in the category $C$.

The following is evident by universality.

\begin{proposition}\label{prop:lambda-s-monic}
Let $H\: R \to S$ be a profunctor.
Any limit heteromorphism $\lambda_s\: \lim s \to s$ in $H$ is monic.\qed
\end{proposition}

Dually, any colimit heteromorphism $\gamma_r\: r \to \colim r$ is epic.

Next, we define the naturality of a family of heteromorphisms, to which we will reduce the naturality of the family of arrows in the theorem's statement.

\begin{definition}
Let $H\: R \to S$ be a profunctor, and let $\rho\: C \to R$ and $\sigma\: C \to S$ be functors.
A family of heteromorphisms $\eta_c\: \rho(c) \to \sigma(c)$ in $H$ over objects $c \in C$ is \emph{natural} if the heteromorphic diagram
$$\begin{tikzcd}[column sep=large]
\rho(c) \arrow[r, "\eta_c"] 
        \arrow[d, "\rho(a)"] &
\sigma(c) \arrow[d, "\sigma(a)"] \\
\rho(c') \arrow[r, "\eta_{c'}"] &
\sigma(c')
\end{tikzcd}$$
commutes for each arrow $a\: c \to c'$ in $C$.
\end{definition}

As in the case of natural transformations, we may colloquially say that the heteromorphism $\eta_c$ is \emph{natural} in $c$ to denote that the family $(\eta_c \mid c \in C)$ is natural.

\begin{proposition}\label{prop:choco}
Let $H\: R \to S$ be a profunctor.
Let $\rho\: C \to R$ and $\sigma\: C \to S$ be functors such that the limit of $\sigma(c)$ in $H$ exists for each $c \in \Ob{C}$.
Then a family of heteromorphisms
$(\eta_c\: \rho(c) \to \sigma(c) \mid c \in \Ob{C})$ is natural if and only if the associated arrows $\overline{\eta_c}\: \rho(c) \to \lim \sigma(c)$ in $R$ are natural in $c$.
\end{proposition}

\begin{proof}
Let $a\: c \to c'$ be an arrow in $C$.
Consider the heteromorphic diagram
$$\begin{tikzcd}
\rho c
    \arrow[r, "\overline{\eta_{\sigma c}}"]
    \arrow[d, "\rho a"] &
\lim \sigma c
    \arrow[r, "\lambda_{\sigma c}"]
    \arrow[d, "\lim \sigma a"] &
\sigma c
    \arrow[d, "\sigma a"] \\
\rho c'
    \arrow[r, "\overline{\eta_{\sigma c'}}"] &
\lim \sigma c'
    \arrow[r, "\lambda_{\sigma c'}"] &
\sigma c'
\end{tikzcd}$$
whose right square commutes as the heteromorphism $\lambda_s\: \lim s \to s$ is natural in objects $s \in S$ (Corollary \ref{cor:limit-functor}).
Note that the horizontal composites are the heteromorphisms $\eta_c\: \rho c \to \sigma c$.
Since the lower-right heteromorphism $\lambda_{\sigma c'}\: \lim \sigma c' \to \sigma c'$ is monic by Proposition \ref{prop:lambda-s-monic}, the left square commutes if and only if the whole rectangle commutes.
This proves the proposition.
\end{proof}

Next we establish the naturality of a family of heteromorphisms to which we can via the last proposition reduce the naturality of the family of arrows in question.
This is an immediate consequence of the extranaturality axiom of indexed profunctors:

\begin{lemma}\label{lem:hc-eta-natural}
Let $H\: R \to S$ be an indexed profunctor over a 2-category $\X$.
Let $C$ be a category, $X,Y \in \X$ objects and $\phi\: C \to \X(X,Y)$ a functor.
Let $h\: r \to s$ be a heteromorphism in $H_X$.
Then the heteromorphism
$$\phi(c)_*(h)\: \phi(c)_*(r) \to \phi(c)_*(s)$$
in $H_Y$ is natural in $c \in \Ob(C)$.
\end{lemma}

\begin{proof}
The heteromorphism in question is natural in $\phi(c)$, an object in the category $\X(X,Y)$, by the extranaturality axiom of $H$ (see Remark \ref{rem:extranaturality-pointwise}).
In particular, by the functoriality of $\phi$, it is natural in $c$.
\end{proof}

We are now ready to deduce Theorem \ref{thm:naturality-of-preservator}.
The heteromorphism
    $$\phi(c)_*(\lambda_s)\: \phi(c)_*(\lim s) \to \phi(c)_*(s)$$
in $H_Y$ is natural in $c$ by Lemma \ref{lem:hc-eta-natural}.
Therefore, the corresponding arrow
    $$\phi(c)_*(\lim s) \to \lim \phi(c)_*(s)$$
in $R_Y$ is natural in $c$ by Proposition \ref{prop:choco}.
This proves the theorem.

\section{Right adjoint 1-cells preserve limits}
\label{sec:right-adjoints-preserve-limits}

In this section, we will prove an indexed-profunctorial generalisation of the fact that right adjoint functors preserve conical limits.
First, consider the following notion of adjointness relative to an indexed profunctor.

\begin{definition}
Let $H\: R \to S$ be an indexed profunctor over a 2-category $\X$.
We say that a 1-cell $g\: Y \to X$ in $\X$ is right \emph{adjoint} to a 1-cell $f\: X \to Y$ in $\X$ \emph{relative to} $H$ if there is a bijection
$H_Y(f_*r,s) \cong H_X(r,g_*s)$
that is natural in objects $r \in R_X$ and $s \in S_Y$.
\end{definition}

Note that if $\X = \Cat$, then adjointness with respect to the indexed profunctor $\Hom\: \id_\Cat \to \id_\Cat$ is the usual `hom-isomorphism' adjointness between functors in $\Cat$.

\begin{theorem}\label{thm:adjoint}
Let $H\: R \to S$ be an indexed profunctor over a 2-category $\X$.
If a 1-cell $g\: Y \to X$ in $\X$ is right adjoint to a 1-cell $f\: X \to Y$ in $\X$ in the 2-categorical sense, then $g$ is canonically right adjoint to $f$ relative to $H$.
\end{theorem}

That is, there is a canonical bijection
    $$H_Y(f_*r, s) \cong H_X(r, g_*s)$$
that is natural in objects $r \in R_X$ and $s \in S_Y$.
Specifically, consider the diagram
\begin{equation}\label{eq:jeju}
\begin{tikzcd}[column sep=large]
H_X(r, g_*s) \arrow[r, "f_*"] &
H_Y(f_*r, f_*g_*s) \arrow[d, "((\epsilon_*)_s)_*"] \\
H_X(g_*f_*r, g_*s) \arrow[u, "(\eta_*)_r^*"] &
H_Y(f_*r, s) \arrow[l, "g_*"] 
\end{tikzcd}
\end{equation}
The theorem is claiming that the composites $\rdsh$ and $\lush$ are natural in $r$ and $s$, and that they are mutually inverse.

\begin{proof}
The naturality is immediate from the fact each of the four sides of \eqref{eq:jeju} is natural in $r$ and $s$.
In what follow, I will argue that the circuit from and to $H_X(r, g_*s)$ is an identity.
The circuit from and to $H_Y(f_*r,s)$ can be shown to be an identity by a symmetric argument.

The circuit \eqref{eq:jeju} from and to $H_X(r,g_*s)$ coincides the circuit
\begin{equation*}\label{eq:adfds}
\begin{tikzcd}[column sep=7em]
H_X(r, g_*s) \arrow[r, "f_*"] &
H_Y(f_*r,f_*g_*s) \arrow[d, "g_*"] \\
H_X(g_*f_*r,g_*s) \arrow[u, "(\eta_*)_r^*"] &
H_X(g_*f_*r,g_*f_*g_*s) \arrow[l, "(g_*(\epsilon_*)_s)_*"]
\end{tikzcd}
\end{equation*}
by the naturality of $g_*\: H_Y \to H_X(g_*,g_*)$.
This circuit in turn clearly coincides the outer circuit of the diagram
\begin{equation*}
\begin{tikzcd}[column sep=7.5em]
H_X(r,g_*s)
    \arrow[r, "f_*"]
    \arrow[d, "((\eta_*)_{g_*s})_*", xshift=0.7ex] &
H_Y(f_*r,f_*g_*s)
    \arrow[d, "g_*"] \\
H_X(r,g_*f_*g_*s)
    \arrow[u, "(g_*(\epsilon_*)_s)_*", xshift=-0.7ex] &
H_X(g_*f_*r,g_*f_*g_*s)
    \arrow[l, "(\eta_*)_r^*"].
\end{tikzcd}
\end{equation*}
Since the inner square commutes by the extranaturality of $H$ at $\eta\: \id_X \to gf$, and since the vertical roundtrip at $H_X(r,g_*s)$ is identity as follows
    $$(g_*(\epsilon_*)_s)_* \circ ((\eta_*)_{g_*s})_* = (((g \epsilon)_*)_s)_* \circ (((\eta g)_*)_s)_* = \id$$
by the triangle identity $g \epsilon \circ \eta g = \id$, the outer circuit is an identity.
This proves the theorem.
\end{proof}

\begin{corollary}\label{cor:adjoint}
The right adjoint 1-cell $g$ preserves limits of type $H$.
\end{corollary}

Dually, the left adjoint 1-cell $f$ preserves colimits of type $H$.

\begin{proof}
Let $s \in S_Y$ be a convergent object.
We need to prove that the heteromorphism $g_*(\lambda_s)\: g_*(\lim s) \to g_*(s)$ in $H_X$ is a limit.
It suffices to show that the chain of natural (in $r$) bijections
\begin{equation*}\begin{tikzcd}[column sep=1.2em]
R_X(r,g_*(\lim s)) \arrow[r, "\cong"] &
R_Y(f_*r,\lim s) \arrow[r, "\cong"] &
H_Y(f_*r,s) \arrow[rrr, "\cong", "(\eta_*)_r^* \circ g_*" swap] &&&
H_X(r, g_*s)
\end{tikzcd}\end{equation*}
at $r = g_*(\lim s)$ sends $\id_{g_*(\lim s)}$ to $g_*(\lambda_s)$.
By feeding $\id_{g_*(\lim s)}$ to the chain, we see that this amounts to claiming that
\begin{equation}\label{eq:night}
    g_*(\lambda_s) = ((\eta_*)_{g_*(\lim s)}^* \circ g_* \circ (\epsilon_*)_{\lim s}^*)(\lambda_s).
\end{equation}
Consider the diagram
\begin{equation*}\begin{tikzcd}[column sep=6em]
H_Y(\lim s, s)
    \arrow[r, "g_*"]
    \arrow[d, "(\epsilon_*)_{\lim{s}}^*"] &
H_X(g_*(\lim s), g_*s)
    \arrow[d, xshift=-0.7ex, "(g_*(\epsilon_*)_{\lim s})^*" swap] \\
H_Y(f_*g_*(\lim s), s)
    \arrow[r, "g_*"] &
H_Y(g_*f_*g_*(\lim s), g_*s)
    \arrow[u, xshift=0.7ex, "(\eta_*)_{g_*(\lim s)}^*" swap]
\end{tikzcd}\end{equation*}
Since the inner square commutes by the naturality of $g_*\: H_Y \to H_X(g_*,g_*)$, and since the vertical roundtrip at the upper-right corner is an identity by the triangle identity, we have that the two ways of getting from the upper-left to the upper-right corner coincide.
This proves \eqref{eq:night} and hence the corollary.
\end{proof}

\section{Why fully faithful functors reflect limits}\label{sec:fully-faithful-reflect}

In this section, we give a profunctorial (and corollarially indexed-profunct\-orial) analysis of the following well-known sufficient conditions for a functor to reflect or preserve conical (co)limits.

\begin{theorem}[E.g.\ \textcite{riehl_category_2017}, Lemmas 3.3.5 and 3.3.6]\label{thm:ff-implies-reflect}
Let $f\: X \to Y$ be a functor between categories.
\begin{enumerate}[1.]
    \item If $f$ is fully faithful, then $f$ reflects conical limits.
    \item If $f$ is essentially surjective and fully faithful, then $f$ preserves conical limits.
\end{enumerate}
Dually:
\begin{enumerate}[resume, label=\arabic*.]
    \item If $f$ is fully faithful, then $f$ reflects conical colimits.
    \item If $f$ is essentially surjective and fully faithful, then $f$ preserves conical colimits.
\end{enumerate}
\end{theorem}

Recall that a 1-cell $f\: X \to Y$ (such as the functor $f$ in the last statement) in the base 2-category of an indexed profunctor $H\: R \to S$ determines the morphism of profunctors 
    $$(R_f,S_f,H_f)\: (R_X,S_X,H_X) \to (R_Y,S_Y,H_Y),$$
and that such a 1-cell $f$ reflects or preserves (co)limits by definition if the profunctor morphism $(R_f,S_f,H_f)$ does so.
Under this association, the following theorem is a profunctorial generalisation of Theorem \ref{thm:ff-implies-reflect}.
After we prove this general theorem, we will gradually consider its indexed-profunctorial special cases, and derive in turn from them Theorem \ref{thm:ff-implies-reflect}.

\begin{theorem}\label{thm:ff-mono-reflect-profunctorial}
Let $(\gamma,\delta,\nu)\: (A,B,M) \to (C,D,N)$ be a morphism of profunctors.
\begin{enumerate}[1.]
    \item If either $\gamma$ is fully faithful and $\nu$ is a mono, or $\gamma$ is splitly full and $\nu$ is an iso, then $(\gamma,\delta,\nu)$ reflects limits.
    
    \item Suppose $\gamma$ is essentially surjective.
    If either $\gamma$ is fully faithful and $\nu$ is a split epi, or $\gamma$ is splitly faithful and $\nu$ is an iso, then $(\gamma,\delta,\nu)$ preserves limits.
\end{enumerate}
Dually:
\begin{enumerate}[resume, label=\arabic*.]
    \item If either $\delta$ is fully faithful and $\nu$ is a mono, or $\delta$ is splitly full and $\nu$ is an iso, then $(\gamma,\delta,\nu)$ reflects colimits.
    
    \item Suppose $\delta$ is essentially surjective.
    If either $\delta$ is fully faithful and $\nu$ is a split epi, or $\delta$ is splitly faithful and $\nu$ is an iso, then $(\gamma,\delta,\nu)$ preserves colimits.
\end{enumerate}
\end{theorem}

\begin{proof}
We need only to prove 1.\ and 2., since 3.\ and 4.\ respectively follow by duality.

1.
Let $m\: a \to b$ in $M$ such that $\nu(m)\: \gamma(a) \to \delta(b)$ in $M'$ is a limit.
We will prove that $m$ is a limit.
Let $a' \in A$ be an object, and consider the square
$$\begin{tikzcd}[column sep=huge]
A(a',a)
    \arrow[r, "(-)^*(m)"]
    \arrow[d, "\gamma" swap] &
M(a',b)
    \arrow[d, "\nu"] \\
C(\gamma{a'},\gamma{a})
    \arrow[r, "(-)^*(\nu{m})" swap, "\cong"] &
N(\gamma{a'},\delta{b}),
\end{tikzcd}$$
in $\SET$, whose upper horizontal arrow we need to show to be a bijection.
By hypotheses, the lower horizontal arrow is an iso, and either the left vertical arrow is an iso and the right vertical arrow is a mono, or the left vertical arrow is a split epi and the right vertical arrow is an iso.
Therefore, it suffices to prove that this square commutes, for then it follows by general nonsense that all arrows in it are isos.
But the commutativity of the square applied to an arrow $r \in A(a',a)$ is precisely the true condition that the naturality square
$$\begin{tikzcd}[column sep=large]
M(a,b)
    \arrow[r, "\nu_{a,b}"]
    \arrow[d, "r^*" swap] &
N(\gamma{a},\delta{b})
    \arrow[d, "(\gamma{r})^*"] \\
M(a',b)
    \arrow[r, "\nu_{a',b}" swap] &
N(\gamma{a'},\delta{b})
\end{tikzcd}$$
of $\nu\: M \to N(\gamma,\delta)$ in the first variable commutes.
This proves 1.

2.
Let $m\: a \to b$ be a limit in $M$.
We need to prove that $\nu(m)\: \gamma(a) \to \delta(b)$ in $N$ is a limit.
Let $c' \in \Ob{C}$.
By the essential surjectivity, there exists an $a' \in \Ob{A}$ and an isomorphism $\phi\: c' \to \gamma{a'}$.
Consider the diagram
$$\begin{tikzcd}[column sep=huge]
A(a',a)
    \arrow[r, "(-)^*(m)", "\cong" swap]
    \arrow[d, "\gamma" swap] &
M(a',b)
    \arrow[d, "\nu"] \\
C(\gamma{a'},\gamma{a})
    \arrow[r, "(-)^*(\nu{m})" swap]
    \arrow[d, "\phi^*" swap, "\cong"] &
N(\gamma{a'},\delta{b})
    \arrow[d, "\phi^*", "\cong" swap] \\
C(c',\gamma{a})
    \arrow[r, "(-)^*(\nu{m})" swap] &
N(c',\delta{b}),
\end{tikzcd}$$
in $\SET$, whose lowermost horizontal arrow we need to show to be a bijection.
By hypotheses, all arrows marked $\cong$ are isos, and either the upper left vertical arrow is an iso and the upper right vertical arrow is a split epi, or the upper left vertical arrow is a split mono and the upper right vertical arrow is an iso.
The upper square commutes as in 1., and the lower square commutes by the functoriality of $N(-,\delta{b})\: C^\op \to \SET$.
By general nonsense, all arrows in the upper square are isos, and consequently so is the lowermost horizontal arrow.
Therefore $\nu{m}$ is a limit.
This proves 2.\ and the theorem.
\end{proof}

\begin{corollary}\label{cor:ff-mono-reflect}
Let $H\: R \to S$ be an indexed profunctor over a 2-category $\X$ and $f\: X \to Y$ a 1-cell in $\X$.
\begin{enumerate}[1.]
    \item If either $R_f$ is fully faithful and $H_f$ is a mono, or $R_f$ is splitly full and $H_f$ is an iso, then $f$ reflects limits.
    
    \item Suppose $R_f$ is essentially surjective.
    If either $R_f$ is fully faithful and $H_f$ is a split epi, or $R_f$ is splitly faithful and $H_f$ is an iso, then $f$ preserves limits.
\end{enumerate}
Dually:
\begin{enumerate}[resume, label=\arabic*.]
    \item If either $S_f$ is fully faithful and $H_f$ is a mono, or $S_f$ is splitly full and $H_f$ is an iso, then $f$ reflects colimits.
    
    \item Suppose $S_f$ is essentially surjective.
    If either $S_f$ is fully faithful and $H_f$ is a split epi, or $S_f$ is splitly faithful and $H_f$ is an iso, then $f$ preserves colimits.\qed
\end{enumerate}
\end{corollary}

If $H$ is corepresentable resp.\ representable, any requirement involving $H_f$ in the last corollary may be reduced to a requirement on $S_f$ resp.\ $R_f$.
This will follow from the following observation.

\begin{proposition}
Let $M,N\: A \to B$ be profunctors and $\theta\: M \to N$ a natural transformation.
If a functor $\alpha\: C \to A$ is faithful, full or splitly full, then the whisker $\theta(\alpha,1)\: M(\alpha,1) \to N(\alpha,1)$ is a mono, epi or split epi respectively.
\end{proposition}

Dually, if a functor $\beta\: C \to B$ is faithful, full or splitly full, then the whisker $\theta(1,\beta)\: M(1,\beta) \to N(1,\beta)$ is a mono, epi or split epi respectively.

\begin{proof}
This is clear inspecting componentwise.
\end{proof}

Continuing in the notations of Corollary \ref{cor:ff-mono-reflect}:

\begin{corollary}
If $H\: R \to S$ is corepresentable and $S_f$ is faithful, full or splitly full, then $H_f$ is a mono, epi or split epi respectively.
\end{corollary}

Dually, if $H\: R \to S$ is representable and $R_f$ is faithful, full or splitly full, then $H_f$ is a mono, epi or split epi respectively.

\begin{proof}
Let $\sigma\: R \to S$ is a strict indexed functor.
If $S_f\: S_X \to S_Y$ is faithful, full or splitly full, or equivalently, the natural transformation
    $$S_f(1,1)\: S_X(1,1) \to S_Y(S_f,S_f)$$
is a mono, epi or split epi, then by the proposition the natural transformation
    $$S_f(\sigma_X,1)\: S_X(\sigma_X,1) \to S_Y(S_f\sigma_X,S_f) = S_Y(\sigma_YR_f,S_f)$$
is a mono, epi or split epi respectively.
This proves the corollary.
\end{proof}

In case $H$ is corepresentable, this allows us to formulate of the following weaker form of Corollary \ref{cor:ff-mono-reflect}.

\begin{corollary}\label{cor:ff-reflect-corepresentable}
Let $H\: R \to S$ be a corepresentable indexed profunctor over a 2-category $\X$ and $f\: X \to Y$ a 1-cell in $\X$.
\begin{enumerate}[1.]
    \item If either $R_f$ is fully faithful and $S_f$ is faithful, or $R_f$ is splitly full and $S_f$ is fully faithful, then $f$ reflects limits.
    
    \item Suppose $R_f$ is essentially surjective.
    If either $R_f$ is fully faithful and $S_f$ is splitly full, or $R_f$ is splitly faithful and $S_f$ is fully faithful, then $f$ preserves limits.
\end{enumerate}
Dually:
\begin{enumerate}[resume, label=\arabic*.]
    \item If $S_f$ is fully faithful, then $f$ reflects colimits.
    
    \item Suppose $S_f$ is essentially surjective.
    If $S_f$ is fully faithful, then $f$ preserves colimits.\qed
\end{enumerate}
\end{corollary}

Suppose $(R,S,H)$ is the indexed profunctor for conical limits of shape $I$, which by definition is corepresentable.
It is easy to check that if a functor $f\: X \to Y \thickspace (= R_f\: R_X \to R_Y)$ is faithful resp.\ splitly full, then the functor $\CAT(I,f)\: \CAT(I,X) \to \CAT(I,Y) \thickspace (= S_f\: S_X \to S_Y)$ is faithful resp.\ splitly full.
Therefore items 1.\ and 2. of Theorem \ref{thm:ff-implies-reflect} follow from items 1.\ and 2.\ of this corollary.

Dually, if $H$ is representable, then the following is the corresponding weaker form of Corollary \ref{cor:ff-mono-reflect}.

\begin{corollary}\label{cor:ff-reflect-representable}
Let $H\: R \to S$ be a representable indexed profunctor over a 2-category $\X$ and $f\: X \to Y$ a 1-cell in $\X$.
\begin{enumerate}[1.]
    \item If $R_f$ is fully faithful, then $f$ reflects limits.
    
    \item Suppose $R_f$ is essentially surjective.
    If $R_f$ is fully faithful, then $f$ preserves limits.
\end{enumerate}
Dually:
\begin{enumerate}[resume, label=\arabic*.]
    \item If either $S_f$ is fully faithful and $R_f$ is faithful, or $S_f$ is splitly full and $R_f$ is fully faithful, then $f$ reflects colimits.
    
    \item Suppose $S_f$ is essentially surjective.
    If either $S_f$ is fully faithful and $R_f$ is splitly full, or $S_f$ is splitly faithful and $R_f$ is fully faithful, then $f$ preserves colimits.\qed
\end{enumerate}
\end{corollary}

Suppose $(R,S,H)$ is the indexed profunctor for conical colimits of shape $I$, which by definition is representable.
As before, if a functor $f\: X \to Y \thickspace (= S_f\: S_X \to S_Y)$ is faithful resp.\ splitly full, then the functor $\CAT(I,f)\: \CAT(I,X) \to \CAT(I,Y) \thickspace (= R_f\: R_X \to R_Y)$ is faithful resp.\ splitly full.
Therefore items 3.\ and 4. of Theorem \ref{thm:ff-implies-reflect} follow from items 3.\ and 4.\ of this corollary.

\section*{Acknowledgements}

I thank Daniel van Dijk for the discussions, particularly for encouraging and helping me to work out the Kan extension example (Example \ref{eg:Kan}).
His interest made it possible for me to bring this document to completion.
I also thank Herman Stel for his comments on an earlier manuscript.

\printbibliography

\end{document}